\documentclass[12pt, reqno]{amsart}

\usepackage{amssymb, amsrefs}
\usepackage{hyperref}

\usepackage{setspace}
\theoremstyle{plain}
\newtheorem{thm}{Theorem}[section]
\newtheorem{prp}[thm]{Proposition}
\newtheorem{cor}[thm]{Corollary}
\newtheorem{lem}[thm]{Lemma}

\newcommand{\U}{\mathcal{U}}
\newcommand{\R}{\mathbb{R}}
\newcommand{\sph}{\mathbb{S}}
\newcommand{\h}{\hbar}
\newcommand{\bigO}{\mathcal{O}}
\newcommand{\cZ}{\mathcal{Z}}
\newcommand{\bigT}{\Theta}
\newcommand{\inner}[2]{ \left\langle #1,#2 \right\rangle}
\newcommand{\e}{\varepsilon}
\newcommand{\Op}{\mathrm{Op}_\h^w}
\newcommand{\Id}{\mathrm{Id}}
\newcommand{\supp}{\mathrm{supp \,}}
\newcommand{\ess}{\mathrm{ess\,supp \,}}
\newcommand{\vol}{\operatorname{Vol}}
\newcommand{\bra}{ \left\langle}
\newcommand{\ket}{\right\rangle}
\newcommand{\diam}{\operatorname{diam}}

\title[Superscarred quasimodes on flat surfaces]{Superscarred quasimodes on flat surfaces with conical singularities}

\date\today

\begin{document}

\begin{abstract}
We construct a continuous family of quasimodes for the Laplace-Beltrami operator on a translation surface. We apply our result to rational polygonal quantum billiards and thus construct a continuous family of quasimodes for the Neumann Laplacian on such domains with spectral width $\bigO_\e(\lambda^{3/8+\e})$. We show that the semiclassical measures associated with this family of quasimodes project to a finite sum of Dirac measures on momentum space, hence, they satisfy Bogomolny and Schmit's superscar conjecture for rational polygons.
\end{abstract}

\author{Omer Friedland}
\address{Institut de Math\'ematiques de Jussieu - Paris Rive Gauche, Sorbonne Universit\'e - Campus Pierre et Marie Curie, 4 place Jussieu, 75252 Paris, France.}
\email{omer.friedland@imj-prg.fr}

\author{Henrik Uebersch\"ar}
\address{Institut de Math\'ematiques de Jussieu - Paris Rive Gauche, Sorbonne Universit\'e - Campus Pierre et Marie Curie, 4 place Jussieu, 75252 Paris, France.}
\email{henrik.ueberschar@imj-prg.fr}

\thanks{H. U. was supported by the grant ANR-17-CE40-0011-01 of the French National
Research Agency ANR (project SpInQS)}

\maketitle

\section{Introduction}

In this paper we study polygonal quantum billiards. It is well-known \cite{T05,ZK75,Z06} that the classical billiard flow on polygons whose angles are rational multiples of $\pi$ can be lifted to the geodesic flow on a flat Riemann surface with a finite number of conical singularities. Whereas the classical dynamics of such billiards has been studied extensively, very little is known at the rigorous level about the quantum dynamics. 

The quantization of the classical billiard flow on rational polygons leads one to study eigenfunctions $\psi_\lambda$ of the Laplace-Beltrami operator on such polygons. For instance, we could consider the eigenvalue problem with Neumann boundary conditions
$$
(\Delta+\lambda)\psi_\lambda=0, \quad \partial_n\psi_\lambda|_{\partial P}=0
$$
where we denote $\Delta=\partial_x^2+\partial_y^2$ in local Euclidean coordinates and $\partial_n$ denotes the outward normal derivative at the boundary $\partial P$.

In fact such quantum billiards have attracted much attention in the physics literature. They are one of the important examples of the class of pseudo-integrable systems. Such systems are close to integrable from a classical point of view in the sense that a full measure subset of trajectories in phase space will never encounter any conical singularities. Nonetheless such systems display features typical of classically chaotic systems at the quantum level which is due to diffraction of wave packets at the conical singularities.

Bogomolny and Schmit \cite{BS04} have predicted that the Laplace eigenfunctions of rational polygons become localized along a finite number of directions in momentum space in the semiclassical limit, as the eigenvalue tends to infinity – a phenomenon that they dubbed ``superscars''. This observation agrees with the expectation that features of the classical dynamics should emerge in the semiclassical limit: a classical trajectory in a rational billiard can only attain a finite number of directions in momentum space. 

Consider a rational polygon which we denote by $P$. In order to test the localization properties of the Laplace eigenfunctions, we recall that one may associate with an eigenfunction $\psi_\lambda$, a Wigner distribution $d\mu_{\psi_\lambda}$ on the phase space $TP$. We are interested in the semiclassical measures associated with this sequence – the limit points of the sequence of Wigner distributions in the weak-* topology for a suitable space of test functions. 

Marklof and Rudnick \cite{MR12} have investigated the projections of these semiclassical measures on configuration space. They found that a full density subsequence of eigenfunctions equidistributes in configuration space, so the only limit measure along this subsequence is normalized Lebesgue measure. However, very little is known about the projection on momentum space. Based on the work of Bogomolny and Schmit, we expect the projections on momentum space to be a finite sum of Dirac measures on $\sph$. 

In this article we make a first step towards proving the Bogomolny-Schmit superscar conjecture by constructing a continuous family of approximate eigenfunctions, frequently referred to as quasimodes, and we show that the semiclassical measures associated with this family of quasimodes are exactly of the form predicted by Bogomolny and Schmit.

The idea to construct the quasimodes follows a method that has recently been used by Nonnenmacher and Eswarathasan \cite{EN17} in the context of hyperbolic systems. We start with a minimal uncertainty coherent state $\phi_0$ which is localized in configuration and momentum space. As we propagate this localized state with the time-evolution operator $\U_t=e^{it\Delta}$ almost all of the mass of the propagated state $\Phi_\lambda=\U_t\phi_0$ will remain close, up to a certain time-scale, to the classical trajectories that correspond to the subset of initial conditions in $TP$ on which the state $\psi_0$ is localized. 
 
Our quasimode is constructed by averaging the evolved state $\Phi_\lambda$ over a time interval $[0,T]$. One can show that this method gives a quasimode of spectral width $O(1/T)$. In our case the maximal permissible time-scale is limited by the necessity to stay away from conical singularities on the associated flat surface. This yields a spectral width of order $\bigO_\e(\lambda^{3/8+\e})$. This is in stark contrast with the work of Nonnenmacher and Eswarathasan who can only get a spectral width of order $O(\lambda^{1/2}/\log\lambda)$, which is of course due to the fact that we are dealing with pseudo-integrable dynamics. We suspect that the present method cannot improve the spectral width much further by going to time-scales where the wave packet is scattered at the conical singularities, as one enters a new regime whose dynamics is much more complex and similar in character to hyperbolic systems. So a further improvement of the time-scale should only be logarithmic in $\lambda$.

Let us now present the main results of this paper. We consider the Laplace-Beltrami operator $\Delta = \partial_x^2 + \partial_y^2$ in local Euclidean coordinates, on the translation surface $Q$, a compact Riemannian manifold with periodic boundary. The pseudo-differential calculus permits to associate with each $L^2$-normalized function $\psi$, a Wigner distribution $d\mu_\psi(x,\xi)$ on the co-tangent bundle $TQ$, which for any classical observable $a\in C^\infty_c(TQ)$ is defined by the duality (see Section \ref{sec-quasi-trans} for a detailed explanation), that is
$$
\bra \Op(a)\psi,\psi\ket_Q = \int_{TQ}a(x,\xi)d\mu_\psi(x,\xi), 
$$
where $\Op(a)$ denotes the standard Weyl quantization. We denote the restriction of the Wigner distribution to momentum space by 
$$
d\mu_\psi(\xi) = \int_Q d\mu_\psi(x,\xi). 
$$

Our goal is to construct quasimodes $\Lambda_\lambda\in C^2(Q)$ for the Laplacian on $Q$, for which the $L^2$-error is controlled in the following manner
\begin{align} \label{eq-spectral-width}
\frac{\|(\Delta + \lambda)\Lambda_\lambda\|_{L^2(Q)}}{\|\Lambda_\lambda\|_{L^2(Q)}} = \bigO(\lambda^\delta) 
\end{align}
for $\delta<1/2$. We say that $\Lambda_\lambda$ is a quasimode with quasienergy $\lambda$ and spectral width $\lambda^\delta$ for the Laplacian on $Q$.

The following theorem shows that there exists a continuous family of quasimodes which satisfy the above conjecture for any translation surface.

\begin{thm} \label{thm-Q}
Let $\xi_0\in\sph$ and let $\e>0$. There exists a continuous family of quasimodes $\{\Lambda_{\lambda}\}_{\lambda>0}$ for the Laplacian on $Q$ of spectral width $\bigO(\lambda^{3/8 + \e})$ so that 
$$
d\mu_{\Lambda_{\lambda}}(\xi) \xrightarrow{w*} \delta(\xi-\xi_0) \,, \quad \lambda\to\infty.
$$
\end{thm}

As a corollary of this theorem one may construct a family of quasimodes for the Neumann Laplacian on a rational polygon $P$. Indeed, any rational polygon $P$ may be unfolded to a translation surface $Q$ under the action of the dihedral group $D$ of $P$ (see e.g. \cite[Section~1.5]{MT02}). Given a quasimode $\Lambda_\lambda$ on $Q$, we may construct a quasimode on $P$ by the method of images, 
$$
\Psi_\lambda(x) = \sum_{g\in D}\Lambda_\lambda(gx).
$$
In other words, given a rational polygon $P$, we lift the problem to an associated translation surface $Q$, apply Theorem \ref{thm-Q}, and then we fold back to the rational polygon $P$, which yields the following result.

\begin{cor} \label{cor-P}
Let $\xi_0\in\sph$ and $\e>0$. There exists a continuous family of quasimodes $\{\Psi_\lambda\}_{\lambda>0}$ for the Neumann Laplacian on $P$ of spectral width $\bigO(\lambda^{3/8 + \e})$ so that 
$$
d\mu_{\Psi_\lambda}(\xi) \xrightarrow{w*} \frac{1}{|D|}\sum_{g\in D}\delta(\xi-g\xi_0) \,, \quad \lambda\to\infty, 
$$
where $D$ is the dihedral group of $P$, and $g$ is an element in this group.
\end{cor}

This paper is organized as follows. First, we give a short description of translation surfaces and the main properties which we use (this is done in Section \ref{sec-trans}). Then, we construct a Gaussian quasimode on the Euclidean plane. It is done by introducing a Gaussian initial state on the Euclidean plane, and then averaging its evolved state over time. We also compute explicitly its spectral width in terms of special functions (see Section \ref{sec-quasi-euclid}). For a translation surface $Q$ we modify the above construction by taking a cutoff of the Gaussian initial state and averaging its evolved state over a time-window, which depends on certain dynamical assumptions. Using Egorov's theorem we calculate its spectral width and actually we show that it is comparable to the Euclidean one, up to an error which we control (see Section \ref{sec-quasi-trans}). Finally, in Section \ref{sec-proofs} we prove Theorem \ref{thm-Q} and Corollary \ref{cor-P}.

We use the following convention for the Fourier transform and its for a function $f \in L^1(\R^2)$.
\begin{align*}
\widehat f(k) = \frac{1}{{2\pi}} \int_{\R^2} f(x) e^{-ik\cdot x} dx \,, \quad f(x) = \frac{1}{{2\pi}} \int_{\R^2} \widehat f(k) e^{ik\cdot x} dk.
\end{align*}

\section{Brief introduction to translation surfaces} \label{sec-trans}

In this section we briefly discuss the notion of translation surfaces and explain the main properties that we shall use. Let us start with the definition. Let $\mathcal{P} = \{P_1,\dots,P_n\}$ be a finite collection of polygons (not necessarily convex nor rational) in the Euclidean plane. A translation surface is the space obtained by edge identification. First, assume that the boundary of each polygon $P_i$ is oriented so that the polygon lies (say) to the left. If $\{s_1,\dots,s_m\}$ is the collection of all edges in $\mathcal P$, then for any $s_i$ there exists $s_j$ so that they are parallel, of the same length and of opposite orientation. In other words, there exists a nonzero translation vector $\tau_i$ so that 
\begin{align} \label{eq-trans-vec}
s_j = s_i + \tau_i,
\end{align}
and hence $\tau_j = -\tau_i$. Now, consider the space obtained by identifying of all $s_i$ with their corresponding $s_j$ through the map $x\mapsto x + \tau_i$, that is, they are ``glued'' together by a parallel translation.

It is well-known that the billiard flow on a rational polygon $P$ may be lifted to the geodesic flow on an associated translation surface $Q$, which is a flat surface with finitely many conical singularities (this is a classical construction, see e.g. \cite{MT02, T05, ZK75}). Consider now a straight line flow on a translation surface $Q$ starting at $x_0\in Q$ in a direction $\xi_0\in\sph^1$. This flow is obtained by starting at $x_0$ and moving in the direction $\xi_0$ at unit speed for time $t$. Once it hits an edge $s_i$ at a point $x$, it continues at $x + \tau_i$ in the same direction $\xi_0$. It is a parametrized curve, which outside of the singular points, is locally the image of a straight line in the Euclidean plane parametrized by arclength. Thus, the flow on $Q$ is nothing but a bunch of parallel closed intervals in direction $\xi_0$. If a flow arrives at a singularity it is required to stop there. How far can it go? It depends on $t$, as long as it does not meet any singular point. A more precise answer is given by the following result of Zemljakov-Katok \cite{ZK75}.

\begin{prp}
For any given time $t$ there exists a direction $\xi_0\in\sph^1$ so that the flow starting at $x_0\in Q$ at direction $\xi_0$ shall not meet any singular point up to time $t$.
\end{prp}

Note also that any straight line flow on $Q$ can be embedded into the plane as a straight line. Once the flow hits a boundary of $Q$, say, the edge $s_i$, instead of translating the flow according the (corresponding) vector $\tau_i$, one can translate $Q$ in the direction $-\tau_i$, and now consider the flow in this translation copy of $Q$. Repeating this process yields a straight line on the Euclidean plane. For more information about translation surfaces, see for example \cite{Z06}.

\section{Gaussian quasimodes on $\R^2$} \label{sec-quasi-euclid}

\subsection{Construction on $\R^2$}

We introduce a Gaussian initial state on the Euclidean plane in the following way. For $(x_0,\xi_0)\in\R^2\times\sph^1$ we denote
\begin{align} \label{eq-gauss-state}
\varphi_0(x) = \sqrt{\frac{\pi}{\h}}\gamma(\frac{x-x_0}{\h^{1/2}})e^{i \frac{\xi_0\cdot x}{\h}}, 
\end{align}
where $\gamma(x) = \frac{1}{2\pi}e^{-|x|^2/2}$ is the standard Gaussian function. Note that $\|\varphi_0\|_{L^2(\R^2)} = 1$ and a simple calculation confirms that the state $\varphi_0$ is localized in position near $x_0$ on a scale $\h^{1/2}$ and in momentum near $\xi_0/\h$ on a scale $\h^{-1/2}$. Clearly, without lose of generality, we may assume that $x_0 = 0$.

The evolved state of $\varphi_0$ is given by $\U_t \varphi_0$, where $\U_t = e^{it\Delta}$ and $\Delta$ is the usual Euclidean Laplacian on $\R^2$. And, we construct a quasimode by averaging the evolved state over time, that is
\begin{align} \label{eq-quasi-euclid}
\Phi_\lambda(x) = \int_{\R} H(t) e^{it\lambda} \U_t \varphi_0(x)dt, 
\end{align}
where $H(t) = \widetilde H(t/T)$, $\widetilde H \in C_c^\infty(\R)$ with $\supp \widetilde H \subset [-1,1]$, and $T>0$ is a time-scale which depends on $\h$.

\subsection{Spectral width of $\Phi_\lambda$}

We compute the spectral width \eqref{eq-spectral-width} of $\Phi_\lambda$. The fact that $\U_t$ is a unitary operator yields 
\begin{align*}
\|\Phi_\lambda\|_{L^2(\R^2)}^2 & = \inner{ \int_\R H(t) e^{it\lambda} \U_t \varphi_0(x) dt}{ \int_\R H(s) e^{is\lambda} \U_s \varphi_0(x) ds} \\
& = \int_{\R} \int_{\R} H(t) H(s) \inner{\varphi_0}{e^{i(s-t)\lambda} \U_{s-t} \varphi_0} dt ds,
\end{align*}
and a simple change of variables $v = s-t$, $u = s + t$ gives
\begin{align} \label{eq-L2-int}
\nonumber \|\Phi_\lambda\|_{L^2(\R^2)}^2 & = \frac12\int_{\R} \left[ \int_{\R} H(\frac{u-v}{2}) H(\frac{u + v}{2}) du \right] \inner{\varphi_0}{e^{iv\lambda} \U_v \varphi_0} dv \\
& = \frac12 \int_{\R} g(v) \inner{\varphi_0}{e^{iv(\Delta + \lambda)} \varphi_0} dv,
\end{align}
where $g(v) = \int_{\R} H(\frac{u-v}{2}) H(\frac{u + v}{2}) du$ is an even function in $v$. Note that 
$$
g(v) = \int_{\R} \widetilde H(\frac{u-v}{2T}) \widetilde H(\frac{u + v}{2T}) du = T \widetilde g (v/T), 
$$
with $\widetilde g (v) = \int_{\R} \widetilde H(\frac{u-v}{2}) \widetilde H(\frac{u + v}{2}) du$.

The next lemma provides the proper estimate for \eqref{eq-L2-int}.

\begin{lem} \label{lem-L2-int}
Let $T\in\R$ and let $\widetilde g \in C_c^\infty(\R)$ be an even function with support $[-2,2]$. Denote $g(v) = T \widetilde g(v/T)$. We have
\begin{align*}
& \int_{\R} g(v) \inner{\varphi_0}{e^{iv(\Delta + \lambda)} \varphi_0} dv = \frac{\h T}{4} \Bigg[ \sum_{\ell = 0}^N a_\ell (-1)^\ell (\h/T)^{2\ell} J_{2\ell} (2/\h) \\
& \qquad + \frac{\widetilde g^{(2N + 2)}(\xi)}{(2N + 2)!} (\h/T)^{2N + 2} J_{2N + 2} (2/\h) \Bigg] + \bigO(\h^\infty), 
\end{align*}
where the $a_\ell$'s are the Taylor coefficients of $\widetilde g$, and
$$
J_\ell(\frac{2}{\h}) = \int_{0}^{2\pi} q_\ell (\cos\theta) e^{-\frac{2}{\h}(1-\cos \theta)} d\theta = \bigT(\h^{1/2}), 
$$
where $q_\ell (x)$ is a monic polynomial in $x$ of degree $\ell$ where the coefficient of the $k^{th}$ monomial of $q_\ell$ is a polynomial in $\h$ of degree $\ell-k$ and $q_\ell (0) = (-1)^\ell$. The polynomials $q_\ell$'s do not depend on $\widetilde g$.
\end{lem}

\begin{proof}
First, let us study the inner product $\inner{\varphi_0}{e^{iv(\Delta + \lambda)} \varphi_0}$. Note that the initial state $\varphi_0$ (see \eqref{eq-gauss-state}) can also be written as follows ($x_0 = 0$)
$$
\varphi_0(x) = \frac{1}{2\sqrt{\pi\h}} \int_{\R^2} \widehat\gamma(k) e^{ix\cdot (\frac{k}{\h^{1/2}} + \frac{\xi_0}{\h})} dk, 
$$
where $\widehat \gamma(k) = \gamma(k)$ is the inverse Fourier transform. Thus, 
\begin{align*}
e^{iv\Delta} \varphi_0 & = \frac{1}{2\sqrt{\pi\h}} \int_{\R^2} \widehat\gamma(k) e^{iv\Delta} e^{ix\cdot (\frac{k}{\h^{1/2}} + \frac{\xi_0}{\h})} dk \\
& = \frac{1}{2\sqrt{\pi\h}} \int_{\R^2} \widehat\gamma(k) e^{-iv|\frac{k}{\h^{1/2}} + \frac{\xi_0}{\h}|^2} e^{ix\cdot (\frac{k}{\h^{1/2}} + \frac{\xi_0}{\h})} dk, 
\end{align*}
and thus by applying the Fourier transform and its inverse on $\gamma$, and a change of variable with respect to $x$, we get
\begin{align*}
& \inner{\varphi_0}{e^{iv(\Delta + \lambda)} \varphi_0} = \\
& = \frac{1}{4\pi\h} \inner{\int_{\R^2} \widehat\gamma(k) e^{i x \cdot (\frac{k}{\h^{1/2}} + \frac{\xi_0}{\h})} dk}{e^{iv\lambda} \int_{\R^2} \widehat\gamma(k') e^{-iv|\frac{k'}{\h^{1/2}} + \frac{\xi_0}{\h}|^2} e^{ix\cdot (\frac{k'}{\h^{1/2}} + \frac{\xi_0}{\h})} dk'} \\
& = \frac{1}{2\h} \int_{\R^2} e^{iv \left[|\frac{k}{\h^{1/2}} + \frac{\xi_0}{\h}|^2 - \lambda\right]} \widehat\gamma(k') \int_{\R^2} e^{- i x \cdot \frac{k'}{\h^{1/2}}} \left[\frac{1}{2\pi}\int_{\R^2} \widehat\gamma(k) e^{i x \cdot \frac{k}{\h^{1/2}}} dk \right] dx dk' \\
& = \pi \int_{\R^2} e^{iv \left[|\frac{k}{\h^{1/2}} + \frac{\xi_0}{\h}|^2 - \lambda\right]} \widehat\gamma(k') \left[\frac{1}{2\pi}\int_{\R^2} \gamma(x) e^{- i x \cdot k'} dx \right] dk' \\
& = \pi \int_{\R^2} \widehat\gamma(k)^2 e^{iv \left[|\frac{k}{\h^{1/2}} + \frac{\xi_0}{\h}|^2 - \lambda\right]} dk. 
\end{align*}

Recall that $\widehat \gamma (k) = \frac{1}{2\pi} e^{-|k|^2/2}$, $\lambda = \h^{-2}$ and denote by $\theta = \angle(r,\xi_0)$ the angle in the coordinate system $(\xi_0,\xi_0^\perp)$. We have
\begin{align*}
 \int_{\R^2} & \widehat\gamma(k)^2 e^{iv \left[|\frac{k}{\h^{1/2}} + \frac{\xi_0}{\h}|^2 - \lambda\right]} dk = \h \int_{\R^2} \left[\widehat\gamma(\h^{1/2}(k-\frac{\xi_0}{\h}))\right]^2 e^{iv(|k|^2 - \lambda)} dk \\
& = \frac{\h}{(2\pi)^2} \int_{0}^{2\pi} \int_{0}^{\infty} e^{-(\h r^2 + \frac{1}{\h}-2 r \cos \theta)} e^{iv(r^2- \lambda)} r dr d\theta \\
& = \frac{\h}{2(2\pi)^2} \int_{0}^{\infty} \int_{0}^{2\pi} e^{-(\h\rho + \frac{1}{\h}-2\rho^{1/2} \cos \theta)} 
d\theta e^{iv(\rho- \lambda)} d\rho \\
& = \frac{\h}{8\pi^2} \int_{0}^{\infty} F(\rho) e^{iv(\rho- \lambda)} d\rho, 
\end{align*}
where $F(\rho) = \int_{0}^{2\pi} f(\rho,\theta) d\theta$ and $f(\rho,\theta) = e^{-(\h\rho + \frac{1}{\h}-2\rho^{1/2} \cos \theta)}$. Thus, we conclude
\begin{align*}
& \inner{\varphi_0}{e^{iv(\Delta + \lambda)} \varphi_0} = \frac{\h}{8\pi} \int_0^\infty F(\rho) e^{iv(\rho- \lambda)} d\rho = \frac{\sqrt{2\pi}\h}{8\pi} e^{-iv\lambda} \widehat F(-v). 
\end{align*}

Therefore, we get
\begin{align*}
& \int_{\R} g(v) \inner{\varphi_0}{e^{iv(\Delta + \lambda)} \varphi_0} dv = \frac{\sqrt{2\pi} \h T}{8\pi} \int_{-2T}^{2T} \widetilde g(v/T) e^{-iv\lambda} \widehat F(-v) dv. 
\end{align*}

Let $P_N(v)$ be the degree $(2N + 2)$ Taylor polynomial for $\widetilde g$ at $0$. By Taylor's theorem for any $v\in(-2,2)$ we have
$$
\widetilde g(v) = P_N(v) + R_N(v), 
$$
where $R_N(v) = \frac{\widetilde g^{(2N + 2)}(\xi)}{(2N + 2)!}v^{2N + 2}$ for $\xi \in (-2,2)$. Thus, 
\begin{align*}
\int_{-2T}^{2T} & \widetilde g(v/T) e^{-iv\lambda} \widehat F(-v) dv = \int_{-2T}^{2T} \left( P_N(v/T) + R_N(v/T) \right) e^{-iv\lambda} \widehat F(-v) dv \\
& = \int_{\R} \left( P_N(v/T) + R_N(v/T) \right) e^{-iv\lambda} \widehat F(-v) dv \\
& \qquad - \int_{|v| \ge 2T} \left( P_N(v/T) + R_N(v/T) \right) e^{-iv\lambda} \widehat F(-v) dv. 
\end{align*}

The second integral is of order $\bigO(\h^\infty)$, because of the rapid decay of $\widehat{h}$ for $|v| \ge 2T$, this indeed can be seen by Egorov's theorem (or by the Gaussian propagation given in the appendix). The first integral can be written as follows
\begin{align*}
 \int_{\R} & \left( P_N(v/T) + R_N(v/T) \right) e^{-iv\lambda} \widehat F(-v) dv = \sum_{\ell = 0}^N a_\ell \int_{\R} (v/T)^{2\ell} e^{-iv\lambda} \widehat F(-v) dv \\
& \quad + \frac{\widetilde g^{(2N + 2)}(\xi)}{(2N + 2)!} \int_{\R} (v/T)^{2N + 2} e^{-iv\lambda} \widehat F(-v) dv, 
\end{align*}
where the $a_\ell$'s are the Taylor's coefficients of $P_N$ and $\xi\in(-2,2)$. Let us now evaluate the integral 
\begin{align*}
\int_{\R} & v^{2\ell} e^{-iv\lambda} \widehat F(-v) dv = \int_{\R} v^{2\ell} e^{iv\lambda} \widehat F(v) dv \\
& = i^{-2\ell}\frac{d^{2\ell}}{d\rho^{2\ell}} \left\{\int_{\R}\widehat{h}(v)e^{iv\rho}dv\right\}\Big|_{\rho = \lambda} = (-1)^\ell \sqrt{2\pi} F^{(2\ell)}(\lambda). 
\end{align*}

Now, 
\begin{align*}
F^{(\ell)} (\lambda) = \int_{0}^{2\pi}\frac{\partial^\ell}{\partial \rho^\ell} f(\rho,\theta) \Big|_{\rho = \lambda} d\theta = \int_{0}^{2\pi} \left(e^{-(\h\rho + \frac{1}{\h}-2\rho^{1/2} \cos \theta)} \right)^{(\ell)} \Big|_{\rho = \lambda} d\theta,
\end{align*}
and the $\ell^{th}$ derivative of $f(\rho,\theta)$ at $\rho = \lambda$ is
\begin{align*}
 \left(e^{-(\h\rho + \frac{1}{\h}-2\rho^{1/2} \cos \theta)} \right)^{(\ell)} \Big|_{\rho = \lambda} = \h^\ell q_\ell (\cos\theta) e^{-\frac{2}{\h}(1-\cos \theta)}, 
\end{align*}
where $q_\ell (x)$ is a monic polynomial in $x$ of degree $\ell$ where the coefficient of the $k^{th}$ monomial of $q_\ell$ is a polynomial in $\h$ of degree $\ell-k$ and $q_\ell(0) = (-1)^\ell$. Therefore, we get 
\begin{align*}
F^{(\ell)}(\lambda) = \h^\ell \int_{0}^{2\pi} q_\ell (\cos\theta) e^{-\frac{2}{\h}(1-\cos \theta)} d\theta = : \h^\ell J_\ell(\frac{2}{\h}). 
\end{align*}

Note that for $x\gg1$
\begin{align*}
\int_{0}^{2\pi} \cos(n \theta) e^{x\cos \theta} d\theta = 2\pi I_n(x) \sim \sqrt{\frac{2\pi}{x}} e^x, 
\end{align*}
where $I_n(x)$ is the modified Bessel function of the first kind, and thus
\begin{align*}
e^{-2/\h} \int_{0}^{2\pi} \cos(n \theta) e^{\frac{2}{\h} \cos \theta} d\theta \sim \sqrt{\pi\h}, 
\end{align*}
which implies that $J_\ell(2/\h) = \bigT(\h^{1/2})$ (for any $\ell \le N$). Thus, we get
\begin{align*}
& \int_{\R} v^{2\ell} e^{-iv\lambda} \widehat F(-v) dv = (-1)^\ell \sqrt{2\pi} \h^{2\ell} J_{2\ell} (2/\h). 
\end{align*}

Finally, combining the calculations above, we get
\begin{align*}
\int_{\R} & g(v) \inner{\varphi_0}{e^{iv(\Delta + \lambda)} \varphi_0} dv = \frac{\sqrt{2\pi} \h T}{8\pi} \Bigg[ \sum_{\ell = 0}^N a_\ell \int_{\R} (v/T)^{2\ell} e^{-iv\lambda} \widehat F(-v) dv \\
& \qquad + \frac{\widetilde g^{(2N + 2)}(\xi)}{(2N + 2)!} \int_{\R} (v/T)^{2N + 2} e^{-iv\lambda} \widehat F(-v) dv \Bigg] + \bigO(\h^\infty) 
\end{align*}
which we may be rewritten as
\begin{align*}
\int_{\R} & g(v) \inner{\varphi_0}{e^{iv(\Delta + \lambda)} \varphi_0} dv = \frac{\h T}{4} \Bigg[ \sum_{\ell = 0}^N a_\ell (-1)^\ell (\h/T)^{2\ell} J_{2\ell} (2/\h) \\
& \qquad + \frac{\widetilde g^{(2N + 2)}(\xi)}{(2N + 2)!} (\h/T)^{2N + 2} J_{2N + 2} (2/\h) \Bigg] + \bigO(\h^\infty). 
\end{align*}
\end{proof}

As we shall explain in Section \ref{sub-dyn} (see \eqref{eq-T}), we have $T<\h$. This constraint comes from our dynamical assumptions, and hence, the main contribution, in Lemma \ref{lem-L2-int}, comes from the first Taylor coefficient $a_0 = \widetilde g(0) = 2\|\widetilde H\|_{L^2(\R^2)}^2$ of $\widetilde g$, and thus, we have
$$
\|\Phi_\lambda\|_{L^2(\R^2)}^2 = \bigT_{\widetilde H} \left( T \h^{3/2}\right). 
$$

By integration by parts, $(\Delta + \lambda)\Phi_\lambda$ can be written as follows
\begin{align} \label{delta + lambda}
\nonumber & (\Delta + \lambda)\Phi_\lambda = (\Delta + \lambda) \int_{\R} H(t) e^{it\lambda)} \U_t \varphi_0(x)dt \\
\nonumber & = \int_{\R} H(t) (\Delta + \lambda) e^{it(\Delta + \lambda)} \varphi_0(x)dt = -i \int_{\R} H(t) \frac{d}{dt} e^{it(\Delta + \lambda)} \varphi_0(x)dt \\
& = i \int_{\R} H'(t) e^{it(\Delta + \lambda)} \varphi_0(x)dt.
\end{align}

So, in view of \eqref{eq-quasi-euclid}, it is essentially of the same form as $\Phi_\lambda$ but with $H'(t)$ instead of $H(t)$. Therefore, 
$$
\|(\Delta + \lambda)\Phi_\lambda\|_{L^2(\R^2)}^2 = \frac12 \int_{\R} g(v) \inner{\varphi_0}{e^{iv(\Delta + \lambda)} \varphi_0} dv,
$$
where 
\begin{align*}
& g(v) = \int_{\R} H'(\frac{u-v}{2}) H'(\frac{u + v}{2}) du \\
& = \frac{1}{T^2} \int_{\R} \widetilde H'(\frac{u-v}{2T}) \widetilde H'(\frac{u + v}{2T}) du = \frac{1}{T} \widetilde g (v/T), 
\end{align*}
and $\widetilde g (v) = \int_{\R} \widetilde H'(\frac{u-v}{2}) \widetilde H'(\frac{u + v}{2}) du$. Its first Taylor coefficient is $a_0 = \widetilde g(0) = 2\|\widetilde H'\|_{L^2(\R^2)}^2$, and thus, in a similar manner, by Lemma \ref{lem-L2-int} we get
\begin{align*}
\|(\Delta + \lambda)\Phi_\lambda\|_{L^2(\R^2)}^2 = \bigT_{\widetilde H} \left( \frac{\h^{3/2}}{T}\right). 
\end{align*}

Finally, we conclude 
\begin{cor} \label{cor-spectral-euclid}
The spectral width of $\Phi_\lambda$ is 
\begin{align*}
\frac{\|(\Delta + \lambda)\Phi_\lambda\|_{L^2(\R^2)}}{\|\Phi_\lambda\|_{L^2(\R^2)}} = \bigT(T^{-1}).\end{align*}
\end{cor}

\section{Quasimodes on translation surfaces} \label{sec-quasi-trans}

\subsection{Construction on translation surfaces}

The construction of our quasimodes on translation surfaces is based on the Gaussian initial state \eqref{eq-gauss-state} that we have introduced on the Euclidean plane. We take a cutoff of this state, and then average the evolved state over time.

Let us now consider a translation surface $Q$ and take $x_0\in Q$. Let $\chi\in C^\infty_c(\R_ + )$ be a cutoff function so that $\chi(x) = 1$ for $x \le \tfrac{1}{2}$ and $\chi(x) = 0$ for $x \ge 1$. For $\h\ll1$ and $\e>0$, we may construct $\psi_0\in C^\infty(Q)$ so that
\begin{align} \label{eq-cutoff-state}
\psi_0(x) = \chi(\frac{|x-x_0|}{\h^{1/2-\e}})\varphi_0(x)
\end{align}
in local Euclidean coordinates. The evolved state of $\varphi_0$ is given by $
\U_t^Q \varphi_0$, where $\U_t^Q$ is defined with respect to the Laplace-Beltrami operator on the translation surface $Q$, that is, a compact Riemannian manifold with periodic boundary. We construct a quasimode for the Laplacian on $Q$, by averaging the evolved state over time, that is
\begin{align} \label{eq-quasi-trans}
\Lambda_\lambda = \int_\R H(t)e^{i\lambda t}\U_t^Q \psi_0 dt,
\end{align}
where $H(t) = \widetilde H(t/T)$, $\widetilde H \in C_c^\infty(\R)$ with $\supp \widetilde H \subset [-1,1]$, and $T>0$ is a time-scale which depends on $\h$. 

\subsection{Dynamical assumptions} \label{sub-dyn}

Recall the definition of the Weyl quantization on $\R^2$. Let $\psi\in\sph(\R^2)$ and $a = a(x,\xi)\in C^\infty_c(\R^2\times\R^2)$. We define (see for instance \cite{Zw12})
$$
[\Op(a)\psi](x) = \frac{1}{(2\pi \h)^2} \int_{\R^2} \int_{\R^2} e^{\frac{i}{\h}(x-y)\cdot\xi}a(\frac{x + y}{2},\xi)\psi(y)dy d\xi.
$$

If we choose an observable, which only depends on the position variable $a = a(x)$ we obtain the restriction of the Wigner distribution to position space 
$$
d\mu_{\psi}(y) = \|\psi\|_{L^2(\R^2)}^{-2}|\psi(y)|^2dy.
$$

On the other hand, if $a(x,\xi) = a(\xi)$, depends only on $\xi$, we find
\begin{align*}
[\Op(a) \psi](x) = \frac{1}{(2\pi \h)^2} \int_{\R^2} \int_{\R^2} e^{\frac{i}{\h}(x-y)\cdot \xi} a(\xi) \psi(y) dy d\xi, 
\end{align*}
and thus
\begin{align} \label{op-inner}
\nonumber & \inner{\Op(a) \psi}{\psi} = \int_{\R^2} \left[ \frac{1}{(2\pi \h)^2} \int_{\R^2} \int_{\R^2} e^{\frac{i}{\h}(x-y)\cdot \xi} a(\xi) \psi(y) dy d\xi \right] \overline{\psi(x)}dx \\
\nonumber & = \frac{1}{(2\pi \h)^2} \int_{\R^2} a(\xi) \int_{\R^2} \overline{\psi(x)} e^{\frac{i}{\h}x\cdot \xi} \left[\int_{\R^2} e^{-\frac{i}{\h}y\cdot \xi} \psi(y) dy \right]dx d\xi \\
\nonumber & = \frac{1}{2\pi \h^2} \int_{\R^2} a(\xi) \widehat \psi(\xi/\h) \int_{\R^2} \overline{\psi(x)} e^{\frac{i}{\h}x\cdot \xi} dx d\xi \\
& = \frac{1}{\h^2} \int_{\R^2} a(\xi) \widehat \psi(\xi/\h) \overline{\widehat \psi(\xi/\h)} d\xi = \frac{1}{\h^2} \int_{\R^2} a(\xi) |\widehat \psi(\xi/\h)|^2 d\xi,
\end{align}
which implies that the restriction of the Wigner distribution to momentum space is given by 
$$
d\mu_{\psi}(\xi) = \h^{-2}\|\psi\|_{L^2(\R^2)}^{-2}|\widehat{\psi}(\xi/\h)|^2.
$$

Now, consider the initial state $\psi_0$ as given in \eqref{eq-cutoff-state}. This implies that $d\mu_{\psi_0}(\xi)$ must be localized near $\xi_0$ on a scale of size $\h^{1/2}$, and similarly, that the restriction of the Wigner distribution to position space is localized near $x_0$ on a scale $\h^{1/2}$. It follows that almost all of the mass (expect for a proportion of order $\h^\infty$) of the Wigner distribution associated with the state $\psi_0$ is concentrated inside the set 
\begin{align} \label{eq-Omega}
\Omega_0 = B(x_0,\h^{1/2-\e}) \times B(\xi_0,\h^{1/2-\e}) \subset TQ.
\end{align}

If we now choose an initial direction $\xi_0$ corresponding to an embedded metric cylinder $\cZ_{\xi_0}$ (a family of parallel periodic orbits) of length $L$ and width $\asymp 1/L$, 
then for 
\begin{align} \label{eq-T}
|v| \le T = \bigO(\h^{3/4 + 2\e})
\end{align}
we have that $\phi_{v/\h}\Omega_0 \subset \cZ_{\xi_0}$, where $\Phi_\lambda$ denotes the geodesic flow on $Q$. 

Indeed, this follows from the observation that 
$$
\diam(\pi_Q(\phi_{v/\h}\Omega_0))\asymp \frac{|v|}{\h}\diam(\pi_{\sph^1}\Omega_0) = |v|\h^{-1/2-\e}
$$
where $\pi_X$ denotes the canonical projection on $X$. In order to remain inside the cylinder we must have $\diam(\pi_Q(\phi_{v/\h}\Omega_0))$ is smaller than the width, that is, $|v|\h^{-1/2-\e} = \bigO(\frac{1}{L})$. Moreover, we impose the condition $\frac{|v|}{\h} \le L$ which ensures that our wave packet may not travel further than the length of the periodic cylinder. This then implies $|v|\h^{-1/2-\e} = \bigO(\frac{\h}{|v|})$ and thus $|v| = \bigO(\h^{3/4 + \e/2})$. In particular, the set $\bigcup_{|v| \le T} \phi_{v/\h}\Omega_0$ does not self-intersect on $Q$.

We may lift the metric cylinder $\cZ_{\xi_0}$ to the Euclidean plane by embedding it in a Euclidean cylinder $\widetilde{\cZ}_{\xi_0}\subset \bigsqcup_i Q_i\subset \R^2$ which lies inside a union of disjoint translates $Q_i = \tau_i Q$ where $\tau_i$ is the corresponding translation vector (see \eqref{eq-trans-vec}).

%

We may, therefore, apply the exact version of Egorov's theorem for the Weyl quantization on the Euclidean plane (cf. \cite[Ch.~4]{Ma02}), since $\supp a_0 \circ \phi_{-v/\h} \subset \phi_{v/\h}\Omega_0$ 
$$
\U_v\Op(a_0)\psi_0 = \Op(a_0\circ \phi_{-v/\h})\U_v \psi_0.
$$

This means that the propagation of the cutoff state $\Op(a_0)\psi_0$ may simply be lifted to the cover $\bigsqcup_i Q_i$ and then projected back to $Q$. In other words, $\supp \psi_0$ and $\supp a_0\circ \phi_{-v/\h}$ may trivially overlap, by construction, only around $v = 0$.

\subsection{Spectral width of $\Lambda_\lambda$}

Using Egorov's theorem we calculate the spectral width $\Lambda_\lambda$ and show that it is comparable to the Euclidean one, that is, the quantity given by Corollary \ref{cor-spectral-euclid}, up to a $\bigO(\h^\infty)$-error.

Repeating the same argument of \eqref{eq-L2-int}, we have 
\begin{align} \label{eq-L2-Lambda}
& \|\Lambda_\lambda\|_{L^2(Q)}^2 = \frac12 \int_{\R} g(v) e^{-iv\lambda} \inner{\psi_0}{\U_v^Q \psi_0}_Q dv,
\end{align}
where $g(v) = \int_{\R} H(\frac{u-v}{2}) H(\frac{u + v}{2}) du$. We may 
choose an observable $a_0\in C^\infty(TQ)$ so that $\supp a_0 \subset \Omega_0$ and 
\begin{align} \label{eq-obv}
\int_{TQ} a(x,\xi)^n d\mu_{\psi_0}(x,\xi) = 1 + \bigO(\h^\infty) \,, \quad n = 1,2. 
\end{align}

Moreover, our dynamical assumptions (see Section \ref{sub-dyn}) guarantee that for any $v\in[-2T,2T]$ we have 
$$
\bigcup_{v\in[-2T,2T]}\supp(a_v)|_Q\cap \partial \cZ_{\xi_0,T} = \emptyset, 
$$
where $\cZ_{\xi_0,T} = \bigcup_{j = -M_T}^{N_T}\tau_j Q$, and $\tau_0 = \Id$. We may now decompose our initial state as follows 
$$
\psi_0 = \Op(a_0)\psi_0 + (1-\Op(a_0))\psi_0, 
$$
where $\Op(a_0)\psi_0$ captures all of the mass of $\psi_0$ but $\bigO(\h^\infty)$ and will be propagated inside $\cZ_{\xi_0,T}$, without intersecting the boundary, so that we may apply the exact 
version of Egorov for the Weyl quantization on $\R^2$. 

We split the integral in \eqref{eq-L2-Lambda} into two parts
\begin{align} \label{eq-split}
\nonumber & \int_{\R} g(v) e^{-iv\lambda} \inner{\psi_0}{\U_v^Q \psi_0}_Q dv = \int_{\R} g(v) e^{-iv\lambda} \inner{\psi_0}{\U_v^Q \Op(a_0)\psi_0}_Q dv \\
& \qquad + \int_{\R} g(v) e^{-iv\lambda} \inner{\psi_0}{\U_v^Q (\Id-\Op(a_0))\psi_0}_Q dv. 
\end{align}

The second term is of order $\bigO(\h^\infty)$. Indeed, by Cauchy-Schwarz and $\|\psi_0\|_{L^2(Q)} \le 1$, we have
\begin{align*}
& |\inner{\psi_0}{\U_v^Q (\Id-\Op(a_0))\psi_0}_Q|^2 \le \|\psi_0\|_{L^2(Q)}^2 \cdot \|\U_v^Q (\Id-\Op(a_0))\psi_0\|_{L^2(Q)}^2 \\
& \qquad \le \inner{(\Id-\Op(a_0))\psi_0}{(\Id-\Op(a_0))\psi_0}_Q. 
\end{align*}

And, since $\Id-\Op(a_v)$ is a self-adjoint operator, we obtain
\begin{align*}
& |\inner{\psi_0}{(\Id-\Op(a_0))\psi_0}_Q|^2 \le \inner{(\Id-\Op(a_0))^2\U_v^Q \psi_0}{\psi_0}_Q \\
& = \|\psi_0\|^2_Q - 2 \inner{\Op(a_0) \psi_0}{\psi_0}_Q + \inner{\Op(a_0^2) \psi_0}{\psi_0}_Q \\
& = 1 + \bigO(\h^\infty) - 2\int_{TQ} a(x,\xi) d\mu_{\psi_0}(x,\xi) + \int_{TQ} a(x,\xi)^2 d\mu_{\psi_0}(x,\xi), 
\end{align*}
where we used the property of the Weyl quantization $\Op(a_0)^2 = \Op(a_0^2)$. It follows
\begin{align*}
\inner{\psi_0}{\U_v^Q (\Id-\Op(a_0))\psi_0}_Q = \bigO(\h^\infty). 
\end{align*}

Therefore, we obtain
\begin{align*}
\int_{\R} g(v) e^{-iv\lambda} \inner{\psi_0}{(\Id-\Op(a_v))\U_v^Q \psi_0}_Q dv = \bigO(\h^\infty). 
\end{align*}

The leading term in \eqref{eq-split} can be written as a sum corresponding to the classical trajectory on $\cZ_{\xi_0,T} \subset \R^2$, translated back to the fundamental domain $Q$. Due to our choice of the observable $a$, the evolution operator $\U_v^Q$ doesn't see the boundary and thus locally can be identified with the evolution operator on $\R^2$.
\begin{align*}
& \int_{\R} g(v) e^{-iv\lambda} \inner{\psi_0}{\U_v^Q \Op(a_0)\psi_0(x)}_Q dv \\
& = \int_{\R} g(v) e^{-iv\lambda} \int_Q \psi_0(x)\sum_{j = -M_T}^{N_T} \overline{\U_v \Op(a_0) \psi_0(\tau_j^{-1}x)} dx dv \\
& = \int_{\R} g(v) e^{-iv\lambda} \int_Q \psi_0(x)\sum_{j = -M_T}^{N_T} \overline{\Op(a_v)\U_v \psi_0(\tau_j^{-1}x)} dx dv \\
& = \sum_{j = -M_T}^{N_T} \int_{\R} g(v) e^{-iv\lambda} \int_Q \psi_0(x) \overline{\Op(a_v)\U_v \psi_0(\tau_j^{-1}x)} dx dv. 
\end{align*}

By our dynamical assumptions for any $j\not = 0$ we have
$$
\inner{\psi_0}{\Op(a_v)\U_v \psi_0(\tau^{-1}x)}_Q = 0, 
$$
as their supports are disjoint. Recall that $\tau_0 = \Id$ and thus we have
\begin{align*}
& \int_{\R} g(v) e^{-iv\lambda} \inner{\psi_0}{\U_v^Q \Op(a_0)\psi_0(x)}_Q dv \\
& = \int_{\R} g(v) e^{-iv\lambda} \int_Q \psi_0(x)\overline{\Op(a_v)\U_v \psi_0(x)} dx dv. 
\end{align*}

Denote $\rho_0 = \varphi_0 - \psi_0$. Then, we have
\begin{align*}
\|\rho_0\|_{L^2(\R^2)}^2 
& = \int_{\R^2} |(1-\chi(\frac{|x|}{\h^{1/2-\e}})) \varphi_0(x)|^2 dx \\
& = \int_{\R^2} |(1-\chi(\frac{|x|}{\h^{1/2-\e}})) \sqrt\frac{\pi}{\h} \gamma(\frac{x}{\h^{1/2}}) e^{\frac{i\xi_0 \cdot x}{\h}} |^2 dx \\
& = \frac{\pi}{\h} \int_{\R^2} (1-\chi(\frac{|x|}{\h^{1/2-\e}}))^2 \gamma(\frac{x}{\h^{1/2}})^2 dx \\
& \le \frac{\pi}{\h} \int_{\R^2 \setminus B(0,\frac12\h^{1/2-\e})} \gamma(\frac{x}{\h^{1/2}})^2 dx = \bigO(e^{-\h^{-\e}}),
\end{align*}
where the last inequality follows from the Gaussian concentration inequality. Thus, we get
\begin{align*}
\int_Q & \psi_0(x) \overline{\Op(a_v)\U_v \psi_0(x)} dx = \int_{\R^2} \psi_0(x) \overline{\Op(a_v)\U_v \psi_0(x)} dx \\
& = \int_{\R^2} (\varphi_0 - \rho_0)(x) \overline{\Op(a_v)\U_v (\varphi_0 - \rho_0)(x)} dx \\
& = \inner{\varphi_0}{\Op(a_v)\U_v \varphi_0} - \inner{\varphi_0}{\rho_0} - \inner{\rho_0}{\Op(a_v)\U_v \varphi_0} + \inner{\rho_0}{\rho_0}. 
\end{align*}

Note that the last three terms are bounded in terms of $\rho_0$. Thus, we reduced the problem to evolution over $\R^2$ and its usual inner product with respect to $\varphi_0$, that is
\begin{align*}
& \inner{\psi_0}{\U_v^Q\Op(a_0) \psi_0}_Q = \inner{\varphi_0}{\U_v \Op(a_0)\varphi_0} + \bigO(e^{-h^{-\e}}). 
\end{align*}

The leading term can be split again as follows
\begin{align*}
\inner{\varphi_0}{\Op(a_v)\U_v \varphi_0} = \inner{\varphi_0}{\U_v \varphi_0} + \inner{\varphi_0}{(\Id- \Op(a_v))\U_v \varphi_0}. 
\end{align*}

The first term is known and given by Corollary \ref{cor-spectral-euclid}, as we shall see shortly. The second term we may estimate as above
\begin{align*}
& |\inner{\varphi_0}{\U_v (\Id-\Op(a_0))\varphi_0}|^2 \\
& \le 1 - 2\int_{T\R^2} a(x,\xi) d\mu_{\varphi_0}(x,\xi) + \int_{T\R^2} a(x,\xi)^2 d\mu_{\varphi_0}(x,\xi). 
\end{align*}

By our choice of the observable $a$ in \eqref{eq-obv}, we conclude 
\begin{align*}
\inner{\psi_0}{\U_v (\Id-\Op(a_0))\varphi_0} = \bigO(\h^\infty).
\end{align*}

Combining all the above, we see that the inner product in \eqref{eq-L2-Lambda} can be written as 
\begin{align*}
\inner{\psi_0}{\U_v^Q \psi_0}_Q = \inner{\varphi_0}{\U_v \varphi_0} + \bigO(\h^\infty),
\end{align*}
which, by Corollary \ref{cor-spectral-euclid}, yields the following connection between the $L^2(Q)$-norm of $\Lambda_\lambda$ and the $L^2(\R^2)$-norm of $\Phi_\lambda$:
\begin{align*}
\|\Lambda_\lambda\|_{L^2(Q)}^2 
& = \frac12 \int_{\R} g(v) e^{-iv\lambda} \inner{\varphi_0}{\U_v \varphi_0} dv + \bigO(\h^\infty) \\
& = \|\Phi_\lambda\|_{L^2(\R^2)}^2 + \bigO(\h^\infty) = \bigT_{\widetilde H} \left( T \h^{3/2}\right). 
\end{align*}

Now, we proceed with $\|(\Delta + \lambda)\Lambda_\lambda\|_{L^2(\R^2)}^2$ as we did in \eqref{delta + lambda}. Hence, we have 
\begin{align*}
(\Delta + \lambda)\Lambda_\lambda = i \int_{\R} H'(t) e^{it\lambda} \U_t \psi_0(x)dt.
\end{align*}

Thus, applying Corollary \ref{cor-spectral-euclid} and verbally repeating the above argument, we conclude
\begin{align*}
\|(\Delta + \lambda)\Lambda_\lambda\|_{L^2(\R^2)}^2 = \|(\Delta + \lambda)\Phi_\lambda\|_{L^2(\R^2)}^2 + \bigO(\h^\infty) = \bigT_{\widetilde H} \left( \frac{\h^{3/2}}{T}\right). 
\end{align*}

Finally, we obtain
\begin{cor}
The spectral width of $\Lambda_\lambda$ is 
$$
\frac{\|(\Delta + \lambda)\Lambda_\lambda\|_{L^2(Q)}}{\|\Lambda_\lambda\|_{L^2(Q)}} = \bigO(T^{-1}) = \bigO(\lambda^{3/8+\e}),
$$
in view of \eqref{eq-T}.
\end{cor}

\section{Concluding proofs} \label{sec-proofs}

\subsection{Proof of Theorem \ref{thm-Q}} \label{sec-thm-Q}

We begin by studying the inverse Fourier transform of the Gaussian quasimode given in \eqref{eq-quasi-euclid}. 

\begin{lem} \label{hat-phi}
We have
\begin{align*}
\widehat \Phi_\lambda(\xi/\h) = \h \int_{\R} H(t) e^{\frac{it}{\h^2}(1-|\xi|^2)} \widehat\gamma(\frac{\xi-\xi_0}{\h^{1/2}}) dt. 
\end{align*}
\end{lem}

\begin{proof}
Indeed, 
\begin{align*}
\Phi_\lambda(x) 
& = \int_{\R} H(t) e^{it\lambda} \U_t \varphi_0(x)dt \\
& = \int_{\R} H(t) e^{it\lambda} \int_{\R^2} \widehat\gamma(k) e^{it\Delta} e^{ix\cdot (\frac{k}{\h^{1/2}} + \frac{\xi_0}{\h})} dk dt \\
& = \h \int_{\R} H(t) e^{it\lambda} \int_{\R^2} \widehat\gamma(\h^{1/2}(k-\frac{\xi_0}{\h})) e^{it\Delta} e^{ix\cdot k} dk dt \\
& = \h \int_{\R} H(t) e^{it\lambda} \int_{\R^2} \widehat\gamma(\h^{1/2}(k-\frac{\xi_0}{\h})) e^{-it|k|^2} e^{ix\cdot k} dk dt \\
& = \int_{\R^2} \left[ \h \int_{\R} H(t) e^{it(\lambda-|k|^2)} \widehat\gamma(\h^{1/2}(k-\frac{\xi_0}{\h})) dt \right] e^{ix\cdot k} dk,
\end{align*}
which implies $\widehat \Phi_\lambda(k) = \h \int_{\R} H(t) e^{it(\lambda-|k|^2)} \widehat\gamma(\h^{1/2}(k-\frac{\xi_0}{\h})) dt$. Note that the ess support of $\widehat \Phi_\lambda(\xi/\h)$ is centered in $B(\xi_0,\h^{1/2-\e})$.
\end{proof}

The next proposition shows that the Wigner distribution associated with the normalized quasimode $\Phi_\lambda/\|\Phi_\lambda\|_{L^2(\R^2)}$ converges weakly to a Dirac mass located at $\xi_0$, which is evident on the Euclidean plane.

\begin{prp} \label{op-euclid}
We have
\begin{align*}
\inner{\Op(a) \Phi_\lambda}{\Phi_\lambda} & = \bigO_a(e^{-h^{-\e}}) \|\widehat \Phi_\lambda\|_{L^2(\R^2)}^2 + a(\xi_0) \|\widehat \Phi_\lambda\|_{L^2(\R^2)}^2 \\
& + \bigO_a(\h^{1/2-\e}) \|\widehat \Phi_\lambda\|_{L^2(\R^2)}^2. 
\end{align*}

Moreover,
\begin{align*}
& \lim_{\h\to 0} \frac{\inner{\Op(a) \Phi_\lambda}{\Phi_\lambda}}{\|\Phi_\lambda\|_{L^2(\R^2)}^2} = a(\xi_0) = \int_{\R^2} a(\xi) \delta(\xi-\xi_0) d\xi.
\end{align*}
\end{prp}

\begin{proof}
Note that $\ess \widehat \Phi_\lambda(\xi/\h) \subset B(\xi_0,\h^{1/2-\e}) = : B$. By \eqref{op-inner} we have
\begin{align*}
& \inner{\Op(a) \Phi_\lambda}{\Phi_\lambda} = \frac{1}{\h^2} \int_{\R^2} a(\xi) |\widehat \Phi_\lambda(\xi/\h)|^2 d\xi \\
& = \frac{1}{\h^2} \int_{\R^2 \setminus B} a(\xi) |\widehat \Phi_\lambda(\xi/\h)|^2 d\xi + \frac{1}{\h^2} \int_{B} a(\xi) |\widehat \Phi_\lambda(\xi/\h)|^2 d\xi \\
& = \frac{1}{\h^2} \int_{\R^2 \setminus B} a(\xi) |\widehat \Phi_\lambda(\xi/\h)|^2 d\xi + \frac{1}{\h^2} \int_{B} \left[ a(\xi_0) + \nabla a(\eta_c)\cdot (\xi-\xi_0) \right] |\widehat \Phi_\lambda(\xi/\h)|^2 d\xi,
\end{align*}
where $\eta_c \in B(\xi_0,\h^{1/2-\e})$. For the first integral we have $|\xi-\xi_0| \ge \h^{1/2-\e}$, and thus by Lemma \ref{hat-phi} we get
\begin{align*}
\frac{1}{\h^2} \int_{\R^2 \setminus B} a(\xi) |\widehat \Phi_\lambda(\xi/\h)|^2 d\xi = \bigO_a(e^{-h^{-\e}}) \|\widehat \Phi_\lambda\|_{L^2(\R^2)}^2. 
\end{align*}

For the second we have
\begin{align*}
\frac{1}{\h^2} \int_{B} a(\xi_0) |\widehat \Phi_\lambda(\xi/\h)|^2 d\xi = a(\xi_0) \|\widehat \Phi_\lambda\|_{L^2(\R^2)}^2. 
\end{align*}

For the last one we have
\begin{align*}
 \left|\frac{1}{\h^2} \int_{B} \nabla a(\eta_c)\cdot (\xi-\xi_0) |\widehat \Phi_\lambda(\xi/\h)|^2 d\xi \right| = \bigO_a(\h^{1/2-\e}) \|\widehat \Phi_\lambda\|_{L^2(\R^2)}^2. 
\end{align*}

Altogether, we get
\begin{align*}
\inner{\Op(a) \Phi_\lambda}{\Phi_\lambda} & = \bigO_a(e^{-h^{-\e}}) \|\widehat \Phi_\lambda\|_{L^2(\R^2)}^2 + a(\xi_0) \|\widehat \Phi_\lambda\|_{L^2(\R^2)}^2 \\
& + \bigO_a(\h^{1/2-\e}) \|\widehat \Phi_\lambda\|_{L^2(\R^2)}^2,
\end{align*}
and thus, $\lim_{\h\to 0} \frac{\inner{\Op(a) \Phi_\lambda}{\Phi_\lambda}}{\|\widehat\Phi_\lambda\|_{L^2(\R^2)}^2} = a(\xi_0)$.
\end{proof}

Now, by our previous calculations, for any $x\in Q$ we have
$$
\Lambda_\lambda(x) = \int_\R H(t)e^{it\lambda}\U_t^Q\psi_0(x)dt = \sum_{j = -M_T}^{N_T}\int_\R H(t)e^{it\lambda}\U_t\phi_0(\tau_j^{-1}x) + \bigO(\h^\infty). 
$$

Let $a\in C_c^\infty (TQ)$, and $a = a(\xi)$. We may therefore approximate the matrix element on $Q$ by the matrix element on $\R^2$, where we introduce an error of order $\bigO(e^{-\h^{-\e}})$
$$
\inner{\Op(a)\Lambda_\lambda}{\Lambda_\lambda}_Q = \inner{\Op(a)\Phi_\lambda}{\Phi_\lambda} + \bigO(\h^\infty). 
$$

This implies, in view of Proposition \ref{op-euclid},
$$
\lim_{\h\to0}\frac{\inner{\Op(a)\Lambda_\lambda}{\Lambda_\lambda}_Q}{\|\Lambda_\lambda\|_{L^2(Q)}^2} = \lim_{\h\to0}\frac{\inner{\Op(a)\Phi_\lambda}{\Phi_\lambda}}{\|\Phi_\lambda\|_{L^2(\R^2)}^2} = a(\xi_0),
$$
which shows that the Wigner distribution associated with the normalized quasimode $\Lambda_\lambda/\|\Lambda_\lambda\|_{L^2(Q)}$ converges weakly to a Dirac mass located at $\xi_0$, which concludes the proof of Theorem \ref{thm-Q}.

\subsection{Proof of Corollary \ref{cor-P}} \label{sec-cor-P}

Let $P$ be a polygon with rational angles. Under the action of a finite group of reflections $G$ we may unfold $P$ to a translation surface 
$$
Q = \bigcup_{g\in G}gP. 
$$

The billiard flow on $P$ may then be lifted to the geodesic flow on the flat surface $M_Q$ that is obtained by gluing the parallel edges of $Q$. Given the quasimode $\Lambda_\lambda$ on $Q$ we may now construct a quasimode on $P$ by the method of images
$$
\Psi_\lambda(x) = \sum_{g\in G}\Lambda_\lambda(gx). 
$$

Let us first calculate the $L^2$-norm of $\Psi_\lambda$
\begin{align*}
& \|\Psi_\lambda\|_{L^2(P)}^2 = \|\sum_{g\in G}\Lambda_\lambda(gx)\|_{L^2(P)}^2 = \int_P \sum_{g,g'\in G}\Lambda_\lambda(gx)\overline{\Lambda_\lambda(g'x)} dx \\
& = \sum_{g\in G}\int_P |\Lambda_\lambda(gx)|^2 dx + \sum_{g\neq g'}\int_P\Lambda_\lambda(gx)\overline{\Lambda_\lambda(g'x)} dx 
= \int_Q |\Lambda_\lambda(x)|^2 dx + \bigO(\h^\infty).
\end{align*}

And, noting that $\Delta(\Lambda_\lambda\circ g) = (\Delta\Lambda_\lambda)\circ g$ (since $g$ is a composition of a translation and a rotation), we find
$$
\|(\Delta + \lambda)\Psi_\lambda\|_{L^2(P)} = \|(\Delta + \lambda)\Lambda_\lambda\|_{L^2(Q)} + \bigO(\h^\infty). 
$$

We thus find that $\Psi_\lambda$ is a quasimode of spectral width $\bigO(T^{-1})$, and we proceed to calculate the semiclassical measure
\begin{align*}
& \inner{\Op(a)}{\Psi_\lambda,\Psi_\lambda}_{L^2(P)} = \sum_{g,g'\in G}\inner{\Op(a)}{\Lambda_\lambda(gx),\Lambda_\lambda(g'x)}_{L^2(P)} \\
& = \sum_{g\in G}\inner{\Op(a)}{\Lambda_\lambda(gx),\Lambda_\lambda(gx)}_{L^2(P)} + \bigO(\h^\infty) 
\end{align*}
which equals
\begin{align*}
&\frac{\vol(P)}{\vol(Q)}\sum_{g\in G}\int_{\sph^1}a(\xi)d\mu_{\Lambda_\lambda}(g\xi) + \bigO(\h^\infty) \\
& \qquad = \frac{1}{|G|}\sum_{g\in G}\int_{\sph^1}a(\xi)d\mu_{\Lambda_\lambda}(g\xi) + \bigO(\h^\infty) \longrightarrow \frac{1}{|G|}\sum_{g\in G}a(g\xi_0), \quad \text{as}\;\lambda\to\infty,
\end{align*}
which concludes the proof of Corollary \ref{cor-P}.

\end{document}